\documentclass[12,letterpaper,twoside]{article}
\usepackage[utf8]{inputenc}
\usepackage[english]{babel}
\usepackage{amsmath, amsfonts, amssymb, mathrsfs, amsthm}
\usepackage{enumerate}
\usepackage{graphicx}
\usepackage{color}
\usepackage{mathtools}
\usepackage[left=3cm,right=3cm,top=2.5cm,bottom=2.5cm]{geometry}
\usepackage{amsmath}
\usepackage{amsfonts}
\usepackage{amssymb}
\usepackage{fancyhdr}
\usepackage{verbatim}
\renewcommand{\thepage}{\arabic{page}}
\newtheorem{eje}{{\it \textbf{Example} }}[section]

\newtheorem{teo}{{\it \textbf{Theorem}}}
\newtheorem{lem}{{\it \textbf{Lemma}}}
\newtheorem{rem}{{\it \textbf{Remark}}}

\providecommand{\keywords}[1]
{
  \small	
  \textbf{\textit{Keywords---}} #1
}

\providecommand{\Classification}[1]	
{
  \small	
  \textbf{\textit{Classification---}} #1
}

\title{{\Large Dynamics of tritrophic interaction with volatile compounds in plants}}
\author{{\normalsize Arturo J. Nic-May$^*$ and Eric J. Avila-Vales$^*$}\\
        {\small$^*$ Facultad de Matemáticas, Universidad Autónoma de Yucatán, Anillo Periférico Norte,}\\{\small Tablaje 13615, Mérida, Yucatán C.P. 97119, México} \\
        {\small  E-mail addresses: arturo\_javier\_1559@hotmail.mx, avila@correo.uady.mx} 
}
\date{} 
\begin{document}
\maketitle
\textbf{Abstract:} In this paper we will consider a mathematical model that describes, the tritrophic interaction between plants, herbivores and their natural enemies, where volatiles organic compounds (VOCs) released by plants play an important role. We show positivity and boundedness of the system solutions, existence of positive equilibrium and its local stability, we analyse global stability of positive equilibrium via the geometrical approach of Li and Muldowney. We pay attention to parameters in order to discuss different  types of bifurcations. Finally, we present some numerical simulations to justify our analytical results.\\
\\
\\
\\
\Classification{92D40, 34D23, 34C23.}
    
\newpage
\section{Introduction}
In agronomy, tritrophic interactions between crop, herbivores and their natural
enemies are one of the drivers of the crop yield. Understanding and manipulating
these interactions in order to produce food more sustainably is the basic principle
of biological control of pest [1]. The plants emit a blend of different Volatile Organic
Compounds (VOCs), Some applications of plant VOCs in agriculture are: isoprenoids
emitted by leaves can exert a protective effect against abiotic stresses by quenching
ROS or by strengthening the cell membranes, some VOCs are able to inhibit
germination and growth of plant pathogens in vitro, herbivore repellency and attraction
of herbivores parasitoids on infested plants are probably the most known
capacity of VOCs [2]. For example, when spider mites damage lima beans and
apple plants, they attract predatory mites by ge\-ne\-ra\-ting VOCs [3]. Corn and cotton
plants also propagate volatiles to call hymenopterous parasitoids which demolish
larvae of several Lepidoptera species [4].

The use of the products chemicals in agriculture has caused serious problems
with food safety and environmental pollution. Thus the agriculture is called
to provide new solutions to increase yields while preserving natural resources and
the environment [2]. For this, various models [5,6,7] have addressed on indirect defense
mechanism of plant population (Vocs). Unlike the models proposed, we
consider the attraction constant, due to VOCs.

In this paper, we consider the model proposed in [1], given by three ordinary
differential equations describing the tritrophic interaction between crop, pest
and the pest natural enemy, in which the release of Volatile Organic Compounds
(VOCs) by crop to attract the pest natural enemy is explicitly taken into account.
Our purpose is to perform a more detailed mathematical analysis of the model
proposed that includes an analysis of different types of bifurcations.

The rest of the paper is organized as follows: The model is introduced in Section
2. Positivity and boundedness of solutions of system are given in Section 3.
Dynamical behavior of the system is investigated in Section 4. Bifurcation phenomenon,
is es\-ta\-bli\-shed in Section 5. Numerical examples are presented in Section
6. A brief discussion is presented in Section 7.

\section{Model}
The model of tritrophic interaction among plants, herbivores and carnivores is
described by following three Ordinary Differential Equations:
\begin{eqnarray}
\frac{\text{d}x}{\text{d}t}&=& rx\left(1-\frac{x}{K}\right)-a\frac{xy}{h+x} \nonumber\\
\frac{\text{d}y}{\text{d}t}&=& y\left(ae\frac{x}{h+x}-m-p\frac{z}{l+y}\right) \nonumber\\
\frac{\text{d}z}{\text{d}t}&=& x\left( b+c\frac{y}{k+y}\right)+z\left(pq\frac{y}{l+y}-n\right). 
\end{eqnarray}
Where all the parameters are positive except $b \geq 0$ and $c \geq 0$, and biological significance are given below:
\begin{itemize}
\item $x$ is the crop population size.
\item $y$ is the aphid population size.
\item $z$ is  the aphid-natural enemy population size.
\item $r$ is the crop growth rate.
\item $K$ is the crop carrying capacity.
\item $a$ is the maximal harvesting rate of crop by aphids.
\item $e$ is the crop to aphids conversion (yield).
\item $m$ is the aphids' natural mortality rate. 
\item is the $p$ maximal uptake rate of aphid by aphid-natural enemy. 
\item $h$, $k$ and $l$ are the half saturation constants.
\item $b$ is the attraction constant due to VOCs.
\item $c$ is the enhanced attraction rate of aphid-natural enemy by VOCs released by crops under aphid attack.
\item $q$ is the aphids to aphid-natural enemy conversion (yield).
\item $n$ is the aphid-natural enemy mortality rate.
\end{itemize}
\section{Positivity and boundedness of solutions}
In this section, we shall first show positivity and boundedness of solutions of system
(1). These are very important so far as the validity of the model is related.
We first study the positivity.

\begin{lem}
All solutions $(x(t), y(t), z(t))$ of system (1) with initial value\\
$(x_0 , y_0 , z_0 ) \in \mathbb{R}^3_+$, remains positive for all $t > 0$.
\end{lem}
\begin{proof}
The positivity of $x(t)$ and $y(t)$ can be verified by the equations
\begin{eqnarray}
x(t)&=&x_0 \exp\left(\int_0^t \left[ r-\frac{rx(s)}{K}-a\frac{y(s)}{h+x(s)}\right]\text{d}s\right),\nonumber\\
y(t)&=&y_0 \exp\left(\int_0^t \left[ ae\frac{x(s)}{h+x(s)}-m-p\frac{z(s)}{l+y(s)}\right]\text{d}s\right).\nonumber
\end{eqnarray}
Also if $x(0) = x_0 > 0$ and $y(0) = y_0 > 0$, then $x(t) > 0$ and $y(t)>0$ for all $t > 0$.
The positivity of $z(t)$ can be easily deduced from the third equation of system (1). We observe that
\begin{equation}
\frac{\text{d}z}{\text{d}t}\geq z \left(pq\frac{y}{l+y}-n\right). \nonumber
\end{equation}
Then.
\begin{equation}
z(t)\geq  z_0 \exp\left(\int_0^t \left[ pq\frac{y}{l+y}-n\right]\text{d}s\right). \nonumber
\end{equation}
if $z(0) = z_0 > 0$, then $z(t)>0$ for all $t>0$.
\end{proof}
\begin{lem} All the solutions of system (1) will lie in the region $\Omega=\{(x,y,z) | x \leq K_1, \ ex+y+\frac{1}{q}z \leq \left( er+\frac{b+c}{q}+1 \right)\frac{K_1}{\delta} \}$, where $\delta=\min \{\frac{1}{e}, \ m, \ n \}$ and $K_1=\max \{x_0, K \}$.
\end{lem}
\begin{proof}
Let $(x(t), y(t), z(t))$ be any solution of system (1) with positive initial conditions
$(x_0, y_0, z_0)$ . Since, $\displaystyle \frac{dx}{dt} \leq rx(1 - \frac{x}{K})$, by a standard comparison theorem we have, $\displaystyle \lim_{t\to \infty} \sup x(t)\leq K_1$.\\
Let $N(t)=ex+y+\frac{1}{q}z$, Then
\begin{eqnarray}
\dot{N}&=&e\left(rx\left(1-\frac{x}{K}\right)-a\frac{xy}{h+x} \right)+y\left(ae\frac{x}{h+x}-m-p\frac{z}{l+y}\right)\nonumber \\
&&+\frac{1}{q}\left(x\left( b+c\frac{y}{k+y}\right)+z\left(pq\frac{y}{l+y}-n\right) \right)\nonumber \\
&=&e\left(rx\left(1-\frac{x}{K}\right)\right)-my+\frac{1}{q}\left(x\left( b+c\frac{y}{k+y}\right)-nz \right)\nonumber \\
&\leq& \left( er+\frac{b}{q}+\frac{c}{q} \right) x -my-\frac{n}{q}z\nonumber\\
&=&\left( er+\frac{b+c}{q}+1 \right) x-x -my-\frac{n}{q}z\nonumber\\
&\leq & \left( er+\frac{b+c}{q}+1 \right)K_1-\delta N \nonumber.
\end{eqnarray}
By using the comparison theorem we have $0 \leq N(t) \leq \left( er+\frac{b+c}{q}+1 \right)\frac{K_1}{\delta}$ for t sufficiently large, so all solutions of $(1)$ are ultimately bounded and enter the region $\Omega $.
\end{proof}
\section{Dynamical behavior}
\subsection{Equilibria}
Here we discuss existence condition of interior equilibrium point of system (1). The
system has one trivial equilibrium point (the ecosystem collapse) $E_0 = (0,0, 0)$,
the aphid-free point $E_1 = (x_1, 0, z_1)$. Where,
\begin{center}
$x_1=K$, $z_1= \displaystyle \frac{b}{n}K$
\end{center}
It follows that the point $E_1$ always exists. And coexistence $E^* = (x^*, y^*, z^*)$, where,
\begin{equation}
y^*=\frac{1}{a}\left[r\left(1-\frac{x}{K}\right)(h+x) \right].
\end{equation}
Which is nonnegative only for $0 \leq x \leq K$,
\begin{equation}
z^*=\frac{l+y}{p}\left[ae\frac{x}{h+x}-m \right].
\end{equation}
This function is nonnegative if $aex \geq m(h + x)$.

Then $y^*$ and $z^*$ are nonnegative if and only if $ae>m$ and $\frac{mh}{ae-m}\leq x^* \leq K$. With $x^*$ being determined by the roots of the equation.
\begin{eqnarray}
H(x)&=&x\left( b+c\frac{y}{k+y}\right)+z\left(pq\frac{y}{l+y}-n\right)\nonumber\\
&=&x\left( b+c\frac{\frac{1}{a}\left[r\left(1-\frac{x}{K}\right)(h+x)) \right]}{k+\frac{1}{a}\left[r\left(1-\frac{x}{K}\right)(h+x)) \right] }\right)+\left( \frac{l+\frac{1}{a}\left[r\left(1-\frac{x}{K}\right)(h+x)) \right]}{p}\right)\nonumber \\
&&\times \left[ae\frac{x}{h+x}-m \right]\left(pq\frac{\frac{1}{a}\left[r\left(1-\frac{x}{K}\right)(h+x)) \right]}{l+\frac{1}{a}\left[r\left(1-\frac{x}{K}\right)(h+x)) \right]}-n\right). \nonumber
\end{eqnarray}
note that $H\left(\frac{mh}{ae-m}\right)>0$  and 
\begin{eqnarray}
H(K)&=&bK-l \frac{n}{p}\left(ae\frac{K}{h+K}-m \right).
\end{eqnarray}
If 
\begin{center}
$ae>m$ and $\displaystyle \frac{aeK}{h+K}\geq m+b\frac{Kp}{ln}.$
\end{center}
Then $H(K)\leq 0$ and by its continuity, the function f must have a zero $x^*$ in the interval $[\frac{mh}{ae-m}, K]$.
\subsection{Local stability}
We now study the local stability of $E_0$, $E_1$ and $E^*$ of Model (1).
\begin{teo} \textit{$E_0=(0,0,0)$ is unstable.}
\end{teo}
\begin{proof}
The Jacobian matrix of the model, we get as follows:

\begin{eqnarray}
J:= \left( \begin{matrix}
r-2\frac{rx}{K}-\frac{ay}{h+x}+\frac{axy}{(h+x)^2}  & -\frac{ax}{h+x} &0 \\
\frac{aehy}{(h+x)^2} & \frac{aex}{h+x}-\frac{lpz}{(l+y)^2}-m&-\frac{py}{l+y} \\
b+c\frac{y}{k+y}& \frac{ckx}{(k+y)^2}+\frac{lpqz}{(l+y)^2} & \frac{pqy}{l+y}-n \end{matrix} \right).
\end{eqnarray}
\begin{eqnarray}
J_{E_0}:= \left( \begin{matrix}
r  & 0 &0 \\
0 & -m&0 \\
b& 0 & -n 
\end{matrix} \right).
\end{eqnarray}
The characteristic equation at $E_0$ is
$$(\lambda-r)(\lambda+m)(\lambda+n)=0$$
Since one of the roots of the above equation is positive, $E_0$ is unstable.
\end{proof}
\begin{teo} If $\frac{aeK}{h+K}<\frac{pbK}{nl}+m$, then $E_1=(K,0,\displaystyle \frac{b}{n}K)$ is locally asymptotically stable. If $\frac{aeK}{h+K}>\frac{pbK}{nl}+m$, then $E_1$ is unstable.
\end{teo}
\begin{proof}
\begin{eqnarray}
J_{E_1}:=  \left( \begin{matrix}
-r  & -\frac{aK}{h+K} &0 \\
0 & \frac{aeK}{h+K}-\frac{pbK}{nl}-m& 0 \\
b& \frac{cK}{k}+\frac{pqbK}{nl} & -n
\end{matrix} \right).
\end{eqnarray}
The characteristic equation at $E_0$ is
$$(\lambda+r)(\lambda-\frac{aeK}{h+K}+\frac{pbK}{nl}+m)(\lambda+n)=0.$$
If $\frac{aeK}{h+K}<\frac{pbK}{nl}+m$ then all the roots of the above equation are negative and hence $E_1$ is locally asymptotically stable. If $\frac{aeK}{h+K}>\frac{pbK}{nl}+m$, since one of the roots of
the above equation is positive, then $E_1$ is unstable.\\
\end{proof}
The Jacobian matrix of the model (1) for the equilibrium point $E^*$ is given by
\begin{eqnarray}
J_{E^*}:= \left( \begin{matrix}
A_{11}  & A_{12} &0 \\
A_{21} & A_{22} & A_{23} \\
A_{31}& A_{32} & A_{33} \end{matrix} \right).
\end{eqnarray}
Where,
\begin{eqnarray}
A_{11}&=&r-2\frac{rx^*}{K}-\frac{ay^*}{h+x^*}+\frac{ax^*y^*}{(h+x)^2}  \nonumber \\
A_{12}&=&-\frac{ax^*}{h+x^*}<0 \nonumber \\
A_{21}&=&\frac{aehy}{(h+x^*)^2}>0 \nonumber\\
A_{22}&=&\frac{aex^*}{h+x^*}-\frac{lpz^*}{(l+y^*)^2}-m \nonumber\\
A_{23}&=&-\frac{py^*}{l+y^*}<0 \nonumber
\end{eqnarray}
\begin{eqnarray}
A_{31}&=&b+c\frac{y^*}{k+y^*}>0 \nonumber\\
A_{32}&=&\frac{ckx^*}{(k+y^*)^2}+\frac{lpqz^*}{(l+y^*)^2}>0 \nonumber\\
A_{33}&=&\frac{pqy^*}{l+y^*}-n. \nonumber
\end{eqnarray}

The characteristic equation at $E^*$ is
$$\lambda^3 +a_1\lambda^2 +a_2 \lambda+a_3=0.$$
Where
\begin{eqnarray}
a_1 &=&-A_{11}-A_{22}-A_{33}\nonumber\\
a_2&=&A_{11}A_{22}+A_{11}A_{33}+A_{22}A_{33}-A_{12}A_{21}-A_{23}A_{32}  \nonumber\\
a_3&=&-A_{11}A_{22}A_{33}-A_{12}A_{23}A_{31}+A_{12}A_{21}A_{33}+A_{11}A_{23}A_{32}.
\nonumber
\end{eqnarray}
Also
\begin{eqnarray}
a_1 a_2-a_3&=&\left( -A_{11}-A_{22}-A_{33}\right)\left(A_{11}A_{22}+A_{11}A_{33}+A_{22}A_{33}-A_{12}A_{21}-A_{23}A_{32} \right)\nonumber\\
&+&A_{11}A_{22}A_{33}+A_{12}A_{23}A_{31}-A_{12}A_{21}A_{33}-A_{11}A_{23}A_{31}\nonumber \\
&=&-A_{11}A_{22}A_{33}+A_{11}A_{23}A_{32} + -A_{11}\left(A_{11}A_{22}+A_{11}A_{33}+-A_{12}A_{21} \right)\nonumber\\
&&-A_{22}\left(A_{11}A_{22}+A_{11}A_{33}+A_{22}A_{33}-A_{12}A_{21}-A_{23}A_{32} \right)\nonumber\\
&&+A_{33}A_{12}A_{21} -A_{33}\left(A_{11}A_{22}+A_{11}A_{33}+A_{22}A_{33}-A_{23}A_{32} \right)\nonumber\\
&+&A_{11}A_{22}A_{33}+A_{12}A_{23}A_{31}-A_{12}A_{21}A_{33}-A_{11}A_{23}A_{32}\nonumber \\
&=&-A_{11}\left(A_{11}A_{22}+A_{11}A_{33}+-A_{12}A_{21} \right)\nonumber\\
&&-A_{22}\left(A_{11}A_{22}+A_{11}A_{33}+A_{22}A_{33}-A_{12}A_{21}-A_{23}A_{32} \right)\nonumber\\
&&-A_{33}\left(A_{11}A_{22}+A_{11}A_{33}+A_{22}A_{33}-A_{23}A_{32} \right)+A_{12}A_{23}A_{31}.\nonumber
\end{eqnarray}
Now by Routh–Hurwitz criterion, it follows, that all roots of $\lambda^3 +a_1\lambda^2 +a_2 \lambda+a_3$ have negative real parts if and only if $a_i > 0$ for $i = 1, 2, 3$ and $a_1a_2 - a_3 > 0$. From above analysis, we now state the following Remarks.
\begin{rem}
If the interior equilibrium point $E^*$ exists then it is locally asymptotically stable if if the following conditions hold: $a_i > 0$ for $i = 1, 2, 3$ and $a_1a_2 - a_3 > 0$.
\end{rem}
\begin{rem}
If $r+\frac{ax^*y^*}{(h+x^*)^2}<2\frac{rx^*}{K}+\frac{ay^*}{h+x^*}$, $\frac{aex^*}{h+x^*}<\frac{lpz^*}{(l+y^*)^2}+m$ and $\frac{pqy^*}{l+y^*}<n$ and $-A_{11}A_{22}A_{33}-A_{12}A_{23}A_{31}+A_{12}A_{21}A_{33}+A_{11}A_{23}A_{32}>0$, then $E^*$ is locally asymptotically stable.
\end{rem}
\subsection{Global stability}
We now study the global stability of endemic equilibria of model (1). We used a
high-dimensional Bendixson criterion of Li and Muldowney [8].

\begin{teo} Suppose $\frac{aeK}{h+K}>\frac{pbK}{nl}+m$ then system (1) is uniformly persistent.
\end{teo}
\begin{proof}
Suppose $x_1$ is a point in the positive octant and $o(x_1)$ is the orbit through $x_1$ and $\omega $ is 
the omega limit set of the orbit through $x_1$. Note that $\omega(x_1)$ is bounded (Lemma 2). We claim 
that $E_0 \notin \omega(x_1)$. If $E_0 \in \omega(x_1)$ then by Butler-McGehee lemma [9], there exists a 
point $P$ in $\omega(x_1)\cap W^s(E_0)$ (which denotes stable manifold of $E_0$). Note that $o(P)$ lies in 
$\omega(x_1)$ and $W^s(E_0)$ is $\Pi_1=\{(0,y,z)| \ y \geq 0 \ \text{and} \ z \geq 0  \}$. Also if $P\in \Pi_1$, consider the following system
\begin{eqnarray}
\frac{\text{d}y}{\text{d}t}&=&v_1(y,z)= y\left(-m-p\frac{z}{l+y}\right) \nonumber\\
\frac{\text{d}z}{\text{d}t}&=&v_2(y,z)= z\left(pq\frac{y}{l+y}-n\right).  
\end{eqnarray}
Note that
\begin{itemize}
\item[•] If $y=0$ and $z>0$, then $v_1(y,z)=0$ and $v_2(y,z)<0$.
\item[•] If $z=0$ and $y>0$, then $v_1(y,z)<0$ and $v_2(y,z)=0$.  
\item[•] If $pq\leq n$, $z>0$ and $y>0$, then $v_1(y,z)<0$ and $v_2(y,z)<0$.  
\item[•] If $pq> n$, $z>0$ and $y>0$, then $v_1(y,z)<0$. Also, if $y>\frac{nl}{pq-n}$ then $v_2(y,z)>0$, if $y=\frac{nl}{pq-n}$ then $v_2(y,z)=0$ and if $0<y<\frac{nl}{pq-n}$ then $v_2(y,z)<0$.  
\end{itemize}
Hence the phase portrait of system (9) is shown in Figure 1, then if $P$ is on the positive side of the $y-axis$, then $o(P)$ is the positive side of the $y-axis$, this contradicts that $ \omega (x_1) $ is bounded. If $P$ is on the positive side of the $z-axis$, then $o(P)$ is the positive side of the $z-axis$, this contradicts that $ \omega (x_1) $ is bounded. Hence let $P=(0,P_2,P_3) \in \{(0,y,z)| \ y> 0 \ \text{and} z > 0\}$, then the orbit through $P$ must be unbounded, giving a contradiction.

\begin{figure}[h]
\centering
\includegraphics[scale=.5]{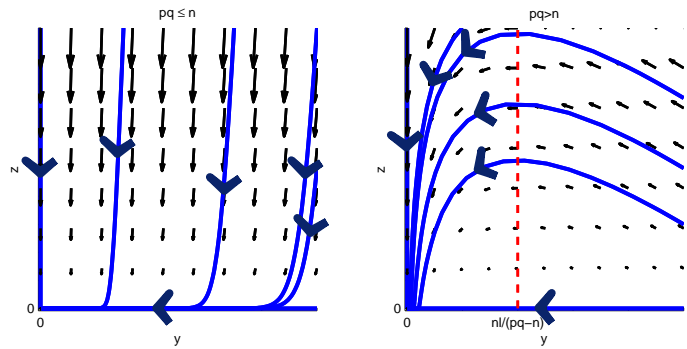}
\caption{The phase portrait of system (9)}
\end{figure}
Next, we show that $E_1 \notin \omega(x_1)$. If $E_1 \in \omega(x_1)$, since $\frac{aeK}{h+K}>\frac{pbK}{nl}+m$, $E_1$ is a saddle point. Then there exists a point $P$ in $\omega(x_1)\cap W^s(E_1)$. Note that $o(P)$ lies in  $\omega(x_1)$ and $W^s(E_1)$ is $\Pi_2=\{(x,0,z)| \ x> 0 \ \text{and} \ z \geq 0  \}$. Also if $P\in \Pi_2$. consider the following system
\begin{eqnarray}
\frac{\text{d}x}{\text{d}t}&=&w_1(x,z)= rx\left(1-\frac{x}{K}\right) \nonumber\\
\frac{\text{d}z}{\text{d}t}&=&w_2(x,z)= bx-nz. 
\end{eqnarray}
Note that

\begin{itemize}
\item[•] If $x>K$ then $w_1(x,z)<0$, if $x=K$ then $w_1(x,z)=0$ and if $0<x<K$ then $w_1(x,z)>0$.
\item[•] If $z>\frac{bx}{n}$ then $w_2(x,z)<0$, if $z=\frac{bx}{n}$ then $w_2(x,z)=0$ and if $z<\frac{bx}{n}$ then $w_2(x,z)>0$.  
\item[•] The linear system of the system (10) has as a stable separatrix, the line $z =\frac{bx}{r+n}$, however, if $z=\frac{bx}{r+n}$, then $w_1(x,z)=rx\left(1-\frac{x}{K}\right)$ and $w_2(x,z)=\frac{brx}{n+r}$, then the stable separatrix surface bends.  
\end{itemize}

\begin{figure}[h]
\centering
\includegraphics[scale=.3]{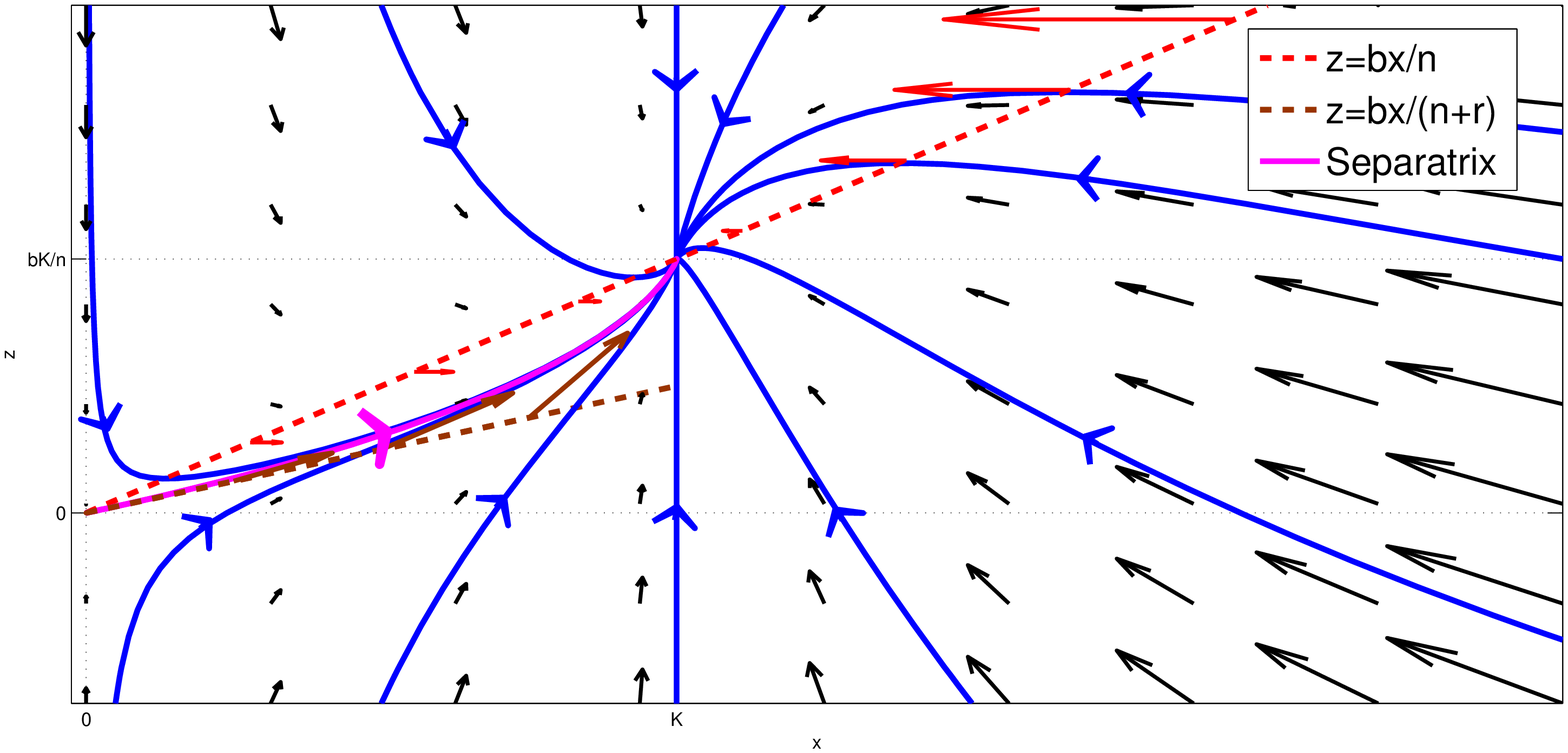}
\caption{The phase portrait of system (10)}
\end{figure}
Hence the phase portrait of system (10) is shown in Figure 2 and hence orbits in the plane emanate from either $E_0$ or an unbounded orbit lies in $\omega(x_1)$, once more a contradiction. There does not exist any equilibria in the two dimensional plane. Thus, $\omega(x_1)$ does not intersect any of the coordinate planes and hence system (1) is persistent. Since (1) is bounded, by main theorem in Butler et al. [10], this implies that the system is uniformly persistent.
\end{proof}
We will make use of the following theorem.
\begin{teo} \text{[8]} Suppose that the system $\dot{x}=f(x)$, with $f:D\subset\mathbb{R}^n\to\mathbb{R}^n$, satisfies the following:
    \begin{itemize}
        \item[\textbf{(H1)}] $D$ is a simply connected open set,
        \item[\textbf{(H2)}] there is a compact absorbing set $K\subset D$,
        \item[\textbf{(H3)}] $x^*$ is the only equilibrium in $D$.
    \end{itemize}
    Then the equilibrium $x^*$ is globally stable in $D$ if there exists a Lozinski\u{\i} measure $\mu_1$ such that
    \begin{equation}
    \limsup_{t\to\infty}\sup_{x_0\in K}\frac1t\int_0^t\mu_1\big(B(x(s,x_0))\big)\,\textup{d}s<0, \nonumber
    \end{equation}
    Where,
    \begin{equation}
    B=Q_fQ^{-1}+Q\frac{\partial f}{\partial x}^{[2]}Q^{-1} \nonumber
    \end{equation}
    And $Q\mapsto Q(x)$ is an ${n\choose2}\times{n\choose2}$ matrix-valued function.
\end{teo}
In our case, system (1) can be written as $\dot{x}=f(x)$ with $f:D\subset\mathbb{R}^3\to\mathbb{R}^3$ and $D$ being the interior of the feasible region $\Omega$. The existence of a compact absorbing set $K\subset D$ is equivalent to proving that (1) is uniformly persistent (Theorem 3). Hence, \textbf{(H1)} and \textbf{(H2)} hold for system (1), and by assuming the uniqueness of the endemic equilibrium in $D$, we can prove its global stability with the aid of Theorem 4.

\begin{teo} If
\begin{itemize}
\item[H1)] There exist positive numbers $\alpha$, $\beta$ and $\zeta$ such that 
$$\max \{ N_{11} +\frac{\alpha}{\beta} N_{12},\frac{\beta}{\alpha} N_{21} + N_{22}+\frac{\beta}{\zeta} N_{23},\frac{\zeta}{\alpha}N_{31} +\frac{\zeta}{\beta}N_{32}+N_{33}\}<0.$$
\item[H2)] $\displaystyle \frac{aeK}{h+K}>\frac{pbK}{nl}+m.$
\end{itemize}
Then $E^*$ is globally stable in $\mathbb{R}^3$.
\end{teo}
\begin{proof}
suppose that $x^*$ is the only equilibrium point in the interior of $\Omega $. By lemma 2 all solution of (1) is bounded, exists a time $T$ such that $x(t)<K_1$, $y(t)\leq M$, and $z(t)\leq qM$ (where $M=\left( er+\frac{b+c}{q}+1 \right)\frac{K_1}{\delta} )$, for $t > T$ and assumption (H2) implies that system (1) is uniformly persistent (Theorem 3) and hence there exists a time $T$ such that $x(t), y(t), z(t) > \eta >0$ for $t > T$.

Starting with the Jacobian matrix $J$ of (1). The Jacobian matrix of the model, we get as follows:

\begin{eqnarray}
J_{E^*}:= \left( \begin{matrix}
a_{11}  & a_{12} &0 \\
a_{21} & a_{22} & a_{23} \\
a_{31}& a_{32} & a_{33} \end{matrix} \right).
\end{eqnarray}
Where,
 \begin{eqnarray}
a_{11}&=&r-2\frac{rx}{K}-\frac{ahy}{(h+x)^2}  \nonumber \\
a_{12}&=&-\frac{ax}{h+x} \nonumber \\
a_{21}&=&\frac{aehy}{(h+x)^2} \nonumber\\
a_{22}&=&\frac{aex}{h+x}-\frac{lpz}{(l+y)^2}-m \nonumber\\
a_{23}&=&-\frac{py}{l+y} \nonumber\\
a_{31}&=&b+c\frac{y}{k+y} \nonumber\\
a_{32}&=&\frac{ckx}{(k+y)^2}+\frac{lpqz}{(l+y)^2} \nonumber\\
a_{33}&=&\frac{pqy}{l+y}-n. \nonumber
\end{eqnarray}
The second additive compound matrix of $J$ is given as follows:
\begin{eqnarray}
M:= \left( \begin{matrix}
M_{11}  & M_{12} &0 \\
M_{21} & M_{22} & M_{23} \\
M_{31}& M_{32} & M_{33} \end{matrix} \right).
\end{eqnarray}
Where,
\begin{eqnarray}
M_{11}&=&r-2\frac{rx}{K}-\frac{ahy}{(h+x)^2}+\frac{aex}{h+x}-\frac{lpz}{(l+y)^2}-m  \nonumber \\ 
M_{12}&=&-\frac{py}{l+y} \nonumber \\
M_{21}&=&\frac{ckx}{(k+y)^2}+\frac{lpqz}{(l+y)^2} \nonumber\\
M_{22}&=&r-2\frac{rx}{K}-\frac{ahy}{(h+x)^2}+\frac{pqy}{l+y}-n \nonumber\\  
M_{23}&=&-\frac{ax}{h+x} \nonumber\\
M_{31}&=&-b-c\frac{y}{k+y} \nonumber\\ 
M_{32}&=& \frac{aehy}{(h+x)^2} \nonumber\\ 
M_{33}&=&\frac{aex}{h+x}-\frac{lpz}{(l+y)^2}-m+\frac{pqy}{l+y}-n. \nonumber 
\end{eqnarray}
Note that, 
\begin{eqnarray}
M_{11}&\leq&r-2\frac{r\eta}{K}-\frac{ah\eta}{(h+K_1)^2}+\frac{aeK_1}{h+\eta}-\frac{lp\eta}{(l+M)^2}-m=N_{11}  \nonumber \\ 
M_{12}&\leq&-\frac{p\eta}{l+\eta}=N_{12} \nonumber \\
M_{21}&\leq&\frac{ck K_1}{(k+\eta)^2}+\frac{lpq^2 M}{(l+\eta)^2}=N_{21} \nonumber\\
M_{22}&\leq&r-2\frac{r \eta }{K}-\frac{ah\eta}{(h+K_1)^2}+\frac{pqM}{l+M}-n=N_{22} \nonumber\\  
M_{23}&\leq&-\frac{a\eta}{h+\eta}=N_{23} \nonumber\\
M_{31}&\leq&-b-c\frac{\eta}{k+\eta}=N_{31} \nonumber\\ 
M_{32}&\leq& \frac{aeh K_1}{(h+\eta)^2}=N_{32} \nonumber\\ 
M_{33}&\leq&\frac{ae K_1}{h+ K_1}-\frac{lp\eta}{(l+M)^2}-m+\frac{pq M}{l+M}-n=N_{33}. \nonumber 
\end{eqnarray}\\
We consider the following norm on $\mathbb{R}^3$.
\begin{equation}
\|z \|=\max \{\alpha |z_1| , \beta |z_2|,\zeta |z_3| \} \ \text{where $\alpha$,$\beta$, $\zeta>0$}. 
\end{equation}
The Lozinskiï measure $\bar{\mu}$ can be evaluated as,  
$$\bar{\mu}(Z)=\inf \{\bar{k}: D_{+} \|z \|\leq \bar{k} \|z \|,\ \text{for all solutions of $z'=Mz$} \}$$
Where $D_+$ is the right-hand derivative. The basic idea of the proof is to the obtain the estimate of the right-hand derivative $D_{+}\|z \|$ of the norm (13), we need to discuss three cases.
\begin{itemize}
\item Case 1: $\alpha |z_1|\geq \beta |z_2|,\zeta|z_3|.$
\end{itemize}  
Then $\|z\|=\alpha |z_1|$.\\
Thus, we have,\\
\begin{eqnarray}
D_{+} \|z \|&=&\alpha \frac{z_1}{|z_1|}z'_1\nonumber\\
&\leq&\alpha M_{11} z_1+\alpha M_{12}z_2 \nonumber\\
&\leq&\left( M_{11} +\frac{\alpha}{\beta} M_{12}\right) \|z\| \nonumber\\
&\leq&\left( N_{11} +\frac{\alpha}{\beta} N_{12}\right) \|z\|. \nonumber
\end{eqnarray}
\begin{itemize}
\item Case 2: $\beta |z_2|\geq \alpha|z_1|, \zeta|z_3|.$
\end{itemize}  
Then $\|z\|=\beta |z_2|$.

Thus, we have,
\begin{eqnarray}
D_{+} \|z \|&=&\beta \frac{z_2}{|z_2|}z'_2\nonumber\\
&\leq& \beta M_{21} z_1+\beta M_{22}z_2+\beta M_{23}z_3 \nonumber\\
&\leq&\left( \frac{\beta}{\alpha} M_{21} + M_{22}+\frac{\beta}{\zeta} M_{23}\right) \|z\| \nonumber\\
&\leq&\left( \frac{\beta}{\alpha} N_{21} + N_{22}+\frac{\beta}{\zeta} N_{23}\right) \|z\|. \nonumber
\end{eqnarray}

\begin{itemize}
\item Case 3: $\zeta |z_3|\geq \alpha |z_1|,\beta |z_2|$.
\end{itemize}  
Then $\|z\|=\zeta |z_3|$.\\
Thus, we have,\\
\begin{eqnarray}
D_{+} \|z \|&=&\zeta \frac{z_3}{|z_3|}z'_3\nonumber\\
&\leq& \zeta M_{31} z_1+\zeta M_{32}z_2+\zeta M_{33} z_3 \nonumber\\
&\leq& \left( \frac{\zeta}{\alpha}M_{31} +\frac{\zeta}{\beta}M_{32}+M_{33}\right) \|z\| \nonumber\\
&\leq& \left( \frac{\zeta}{\alpha} N_{31}+\frac{\zeta}{\beta}N_{32}+N_{33}\right) \|z\|. \nonumber
\end{eqnarray}
Therefore 
$$D_{+} \|z \|\leq L \|z \|.$$
Where:
$$L=\max \{ N_{11} +\frac{\alpha}{\beta} N_{12},\frac{\beta}{\alpha} N_{21} + N_{22}+\frac{\beta}{\zeta} N_{23},\frac{\zeta}{\alpha}N_{31} +\frac{\zeta}{\beta}N_{32}+N_{33}\}<0.$$
So 
\begin{equation}
\limsup_{t \to \infty} \sup_{x_0\in \omega}    \frac{1}{t}\int^t_0 \bar{\mu}(M) ds\leq \limsup_{t \to \infty} \sup_{x_0\in \omega}    \frac{1}{t}\int^t_0 L ds=L<0.   \nonumber
\end{equation}
By LI \& Muldowney[8] and theorem 4, the positive equilibrium point $E^*$ is globally stable in $\mathbb{R}^3_+$.
\end{proof}
\section{Bifurcation}
In this section we discuss various types of bifurcation of system (1) around different steady states.
\begin{teo}If $\frac{2aeh}{(h+x_1)^2}v^{[1]}+\frac{2pz_1}{l^2}-\frac{2p}{l}v^{[3]}\neq 0$, where $v^{[1]}=-\frac{aK}{r(h+K)}$ and $v^{[3]}=\frac{-\frac{baK}{h+K}+\frac{rcK}{k}+\frac{rpqbK}{nl}}{rn}$. Then the system (1) possesses a transcritical bifurcation 
at the equilibrium point $E_1$ 
as the parameter $m$ crosses the critical value $m^*=\frac{aeK}{h+K}-\frac{pbK}{nl}$.
\end{teo}
\begin{proof}
Let $X=(x,y,z)$ and
\begin{eqnarray}
f(X,m)=\left( \begin{matrix}
rx\left(1-\frac{x}{K}\right)-a\frac{xy}{h+x}\\
y\left(ae\frac{x}{h+x}-m-p\frac{z}{l+y}\right)\\
 x\left( b+c\frac{y}{k+y}\right)+z\left(pq\frac{y}{l+y}-n \right)\end{matrix}\right).\nonumber 
\end{eqnarray}
\begin{eqnarray}
f_m(X,m)=\left( \begin{matrix}
0\\
-y\\
0\end{matrix}\right).\nonumber 
\end{eqnarray}

\begin{eqnarray}
Df(X,m)= \left( \begin{matrix}
r-2\frac{rx}{K}-\frac{ay}{h+x}+\frac{axy}{(h+x)^2}  & -\frac{ax}{h+x} &0 \\
\frac{aehy}{(h+x)^2} & \frac{aex}{h+x}-\frac{lpz}{(l+y)^2}-m&-\frac{py}{l+y} \\
b+c\frac{y}{k+y}& \frac{ckx}{(k+y)^2}+\frac{lpqz}{(l+y)^2} & \frac{pqy}{l+y}-n \end{matrix} \right). \nonumber
\end{eqnarray}
\begin{eqnarray}
Df_m(X,m)= \left( \begin{matrix}
0 & 0 &0 \\
0 & -1 &0 \\
0& 0 & 0 \nonumber \end{matrix} \right).
\end{eqnarray}
Then
\begin{eqnarray}
f_m(E_1,m)=\left( \begin{matrix}
0\\
0\\
0\end{matrix}\right).\nonumber 
\end{eqnarray}
  
\begin{eqnarray}
A=Df(E_1,m^*):=  \left( \begin{matrix}
-r  & -\frac{aK}{h+K} &0 \\
0 & 0& 0 \\
b& \frac{cK}{k}+\frac{pqbK}{nl} & -n
\end{matrix} \right).
\end{eqnarray}
$A$  has a simple eigenvalue  $\lambda=0$ with eigenvector $v=(v^{[1]},1,v^{[3]})^T$, where $v^{[1]}=-\frac{aK}{r(h+K)}$ and $v^{[3]}=\frac{-\frac{baK}{h+K}+\frac{rcK}{k}+\frac{rpqbK}{nl}}{rn}$. Also, $A^T$ has an eigenvector $w=(0,1,0)^T$ that co\-rres\-ponds to the eigenvalue $ \lambda =
0$.

Also:
$$w^T[f_m(E_1,m^*)]=0.$$ 
\begin{eqnarray}
w^T[Df_m(X,m)v]=(0,1,0)\left[ \left( \begin{matrix}
0 & 0 &0 \\
0 & -1 &0 \\
0& 0 & 0  \end{matrix} \right) \left( \begin{matrix}
v^{[1]} \\
1 \\
v^{[3]}  \end{matrix} \right) \right]=-1\neq 0. \nonumber
\end{eqnarray}

\begin{eqnarray}
w^T[D^2f(E_1,m^*)(v,v)]&=&(0,1,0)\nonumber\\
&&\times \left( \begin{matrix}
-\frac{2r}{K} v^{[1]}v^{[1]}-\frac{2 a h}{(h+x_1)^2}v^{[1]} \\
\frac{2aeh}{(h+x_1)^2}v^{[1]}+\frac{2pz_1}{l^2}-\frac{2p}{l}v^{[3]}\\
\frac{2 c}{k}v^{[1]}+\frac{2pq}{l}v^{[3]}-\frac{2cx_1}{k^2}-\frac{2pqz_1}{l^2}  \end{matrix} \right)  \nonumber\\
 &=&\frac{2aeh}{(h+x_1)^2}v^{[1]}+\frac{2pz_1}{l^2}-\frac{2p}{l}v^{[3]}\neq 0. \nonumber
\end{eqnarray}
By Sotomayor theorem [11], the system (1) experiences a transcritical bifurcation at the equilibrium
point $E_1$ as the parameter $m$ varies through the bifurcation value $m=m^*$.
\end{proof}

\begin{teo} If $\frac{2aeh}{(h+x_1)^2}v^{[1]}+\frac{2pz_1}{l^2}-\frac{2p}{l}v^{[3]}\neq 0$, where $v^{[1]}=-\frac{aK}{r(h+K)}$ and $v^{[3]}=\frac{-\frac{b^*aK}{h+K}+\frac{rcK}{k}+\frac{rpqb^*K}{nl}}{rn}$. Then the system (1) possesses a transcritical bifurcation  
at the equilibrium point $E_1$ 
as the parameter $b$ crosses the critical 
value $b^*=\left(\frac{aeK}{h+K}-m \right)\frac{nl}{pK }$.
\end{teo}
\begin{proof}
Let $X=(x,y,z)$ and
\begin{eqnarray}
f(X,b)=\left( \begin{matrix}
rx\left(1-\frac{x}{K}\right)-a\frac{xy}{h+x}\\
y\left(ae\frac{x}{h+x}-m-p\frac{z}{l+y}\right)\\
 x\left( b+c\frac{y}{k+y}\right)+z\left(pq\frac{y}{l+y}-n \right)\end{matrix}\right).\nonumber 
\end{eqnarray}
\begin{eqnarray}
f_b(X,b)=\left( \begin{matrix}
0\\
0\\
x\end{matrix}\right).\nonumber 
\end{eqnarray}

\begin{eqnarray}
Df(X,b)= \left( \begin{matrix}
r-2\frac{rx}{K}-\frac{ay}{h+x}+\frac{axy}{(h+x)^2}  & -\frac{ax}{h+x} &0 \\
\frac{aehy}{(h+x)^2} & \frac{aex}{h+x}-\frac{lpz}{(l+y)^2}-m&-\frac{py}{l+y} \\
b+c\frac{y}{k+y}& \frac{ckx}{(k+y)^2}+\frac{lpqz}{(l+y)^2} & \frac{pqy}{l+y}-n \end{matrix} \right). \nonumber
\end{eqnarray}
\begin{eqnarray}
Df_b(E_1,b^*)= \left( \begin{matrix}
0 & 0 &0 \\
0 & -\frac{pK}{nl} &0 \\
1& \frac{pqK}{nl} & 0 \nonumber \end{matrix} \right).
\end{eqnarray}
Then
\begin{eqnarray}
f_b(E_1,b^*)=\left( \begin{matrix}
0\\
0\\
x_1
\end{matrix}\right).\nonumber 
\end{eqnarray}
  
\begin{eqnarray}
A=Df(E_1,b^*):=  \left( \begin{matrix}
-r  & -\frac{aK}{h+K} &0 \\
0 & 0& 0 \\
b^*& \frac{cK}{k}+\frac{pqb^*K}{nl} & -n
\end{matrix} \right).
\end{eqnarray}
$A$  has a simple eigenvalue  $\lambda=0$ with eigenvector $v=(v^{[1]},1,v^{[3]})^T$, where $v^{[1]}=-\frac{aK}{r(h+K)}$ and $v^{[3]}=\frac{-\frac{b^*aK}{h+K}+\frac{rcK}{k}+\frac{rpqb^*K}{nl}}{rn}$. Also,  $A^T$ has an eigenvector $w=(0,1,0)^T$ that co\-rres\-ponds to the eigenvalue $ \lambda =
0$.\\
Also:
$$w^T[f_b(E_1,b^*)]=0.$$ 
\begin{eqnarray}
w^T[Df_b(X,b)v]=(0,1,0)\left[ \left( \begin{matrix}
0 & 0 &0 \\
0 & -\frac{pK}{nl} &0 \\
1& -\frac{pqK}{nl} & 0  \end{matrix} \right) \left( \begin{matrix}
v^{[1]} \\
1 \\
v^{[3]}  \end{matrix} \right) \right]=-\frac{pK}{nl}\neq 0. \nonumber
\end{eqnarray}

\begin{eqnarray}
w^T[D^2f(E_1,b^*)(v,v)]&=&(0,1,0)\nonumber\\
&&\times \left( \begin{matrix}
-\frac{2r}{K} v^{[1]}v^{[1]}-\frac{2 a h}{(h+x_1)^2}v^{[1]} \\
\frac{2aeh}{(h+x_1)^2}v^{[1]}+\frac{2pz_1}{l^2}-\frac{2p}{l}v^{[3]}\\
\frac{2 c}{k}v^{[1]}+\frac{2pq}{l}v^{[3]}-\frac{2cx_1}{k^2}-\frac{2pqz_1}{l^2}  \end{matrix} \right)  \nonumber\\
 &=&\frac{2aeh}{(h+x_1)^2}v^{[1]}+\frac{2pz_1}{l^2}-\frac{2p}{l}v^{[3]}\neq 0. \nonumber
\end{eqnarray}
By Sotomayor theorem [11], the system (1) experiences a transcritical bifurcation at the equilibrium
point $E_1$ as the parameter $b$ varies through the bifurcation value $b=b^*$.
\end{proof}
Let $b=\tilde{b}=\frac{-A_{11}A_{22}A_{33}+A_{12}A_{21}A_{33}+A_{11}A_{23}A_{32}}{A_{12}A_{23}}-\frac{cy^*}{k+y^*}$, $X=(x,y,z)$ and
\begin{eqnarray}
f(X,b)=\left( \begin{matrix}
rx\left(1-\frac{x}{K}\right)-a\frac{xy}{h+x}\\
y\left(ae\frac{x}{h+x}-m-p\frac{z}{l+y}\right)\\
 x\left( b+c\frac{y}{k+y}\right)+z\left(pq\frac{y}{l+y}-n \right)\end{matrix}\right).\nonumber 
\end{eqnarray}
\begin{eqnarray}
f_b(X,b)=\left( \begin{matrix}
0\\
0\\
x\end{matrix}\right).\nonumber 
\end{eqnarray}

\begin{eqnarray}
Df(X,b)= \left( \begin{matrix}
r-2\frac{rx}{K}-\frac{ay}{h+x}+\frac{axy}{(h+x)^2}  & -\frac{ax}{h+x} &0 \\
\frac{aehy}{(h+x)^2} & \frac{aex}{h+x}-\frac{lpz}{(l+y)^2}-m&-\frac{py}{l+y} \\
b+c\frac{y}{k+y}& \frac{ckx}{(k+y)^2}+\frac{lpqz}{(l+y)^2} & \frac{pqy}{l+y}-n \end{matrix} \right). \nonumber
\end{eqnarray}

Then
\begin{eqnarray}
f_b(E^*,\tilde{b})=\left( \begin{matrix}
0\\
0\\
x^*
\end{matrix}\right).\nonumber 
\end{eqnarray}
  
\begin{eqnarray}
A=Df(E^*,\tilde{b}):=  \left( \begin{matrix}
A_{11}  & A_{22} &0 \\
A_{21} & A_{22}& A_{23} \\
A_{31}& A_{32} & A_{33} 
\end{matrix} \right). \nonumber
\end{eqnarray}
Also if $b=\tilde{b}$, then $a_3=0$ and the characteristic equation at $E^*$ is
$$\lambda^3 +a_1\lambda^2 +a_2 \lambda=\lambda(\lambda^2 +a_1\lambda +a_2)=0.$$
Then $A$  has a simple eigenvalue  $\lambda=0$ with eigenvector $v=(v^{[1]},v^{[2]},1)^T$, where $v^{[1]}=-\frac{A_{12}}{A_{11}}$ and $v^{[2]}=\frac{-A_{23}A_{11}}{A_{11}A_{22}-A_{12}A_{21}}$. Also,  $A^T$ has an eigenvector $w=(w^{[1]},w^{[2]},1)^T$, where $w^{[2]}=-\frac{A_{11}A_{32}-A_{31}A_{12}}{A_{11}A_{22}-A_{12}A_{21}}$ and $w^{[1]}=-\frac{A_{21}w^{[2]}}{A_{11}}-\frac{A_{31}}{A_{11}}$, that corresponds to the eigenvalue $ \lambda =0$.\\
Also:
$$w^T[f_b(E^*,b^*)]=x^*\neq 0.$$ 
\begin{eqnarray}
w^T[D^2f(E_1,b^*)(v,v)]&=&(w^{[1]},w^{[2]},1)\nonumber\\
&\times & \left( \begin{matrix}
\left(-\frac{2r}{K}+\frac{2hay^*}{(h+x^*)^3}\right) v^{[1]}v^{[1]}-\frac{2 a h}{(h+x^*)^2}v^{[1]}v^{[2]} \\
-\frac{2haey^*}{(h+x^*)^2}{v^{[1]}}^2+\frac{2aeh}{(h+x^*)^2}v^{[1]}v^{[2]}+\frac{2lpz^*}{(l+y^*)^3}{v^{[2]}}^2-\frac{2pl}{(l+y^*)^2}v^{[2]}\\
\frac{2 ck}{(k+y^*)^2}v^{[1]}v^{[2]}+\frac{2lpq}{(l+y^*)^2}v^{[2]}-\left(\frac{2ckx^*}{(k+y)^3}+\frac{2lpqz^*}{(l+y^3)^3}\right)v^{[2]}v^{[2]}  \end{matrix} \right)  \nonumber\\
 &=&w^{[1]}\left(\left(-\frac{2r}{K}+\frac{2hay^*}{(h+x^*)^3}\right) v^{[1]}v^{[1]}-\frac{2 a h}{(h+x^*)^2}v^{[1]}v^{[2]}\right)\nonumber \\
&&+w^{[2]}\left(-\frac{2haey^*}{(h+x^*)^2}v^{[1]}v^{[1]}+\frac{2aeh}{(h+x^*)^2}v^{[1]}v^{[2]}\right. \nonumber\\
&&\left.+\frac{2lpz^*}{(l+y^*)^3}v^{[2]}v^{[2]}-\frac{2pl}{(l+y^*)^2}v^{[2]}\right)+\frac{2 ck}{(k+y^*)^2}v^{[1]}v^{[2]} \nonumber\\
&&+\frac{2lpq}{(l+y^*)^2}v^{[2]}-\left(\frac{2ckx^*}{(k+y)^3}+\frac{2lpqz^*}{(l+y^3)^3}\right)v^{[2]}v^{[2]}\nonumber
\end{eqnarray}
From above analysis, we now state the following Remark.
\begin{rem}
If $b=\tilde{b}$ and\\
\begin{center}
{\small $w^{[1]}\left(\left(-\frac{2r}{K}+\frac{2hay^*}{(h+x^*)^3}\right) v^{[1]}v^{[1]}-\frac{2 a h}{(h+x^*)^2}v^{[1]}v^{[2]}\right)+w^{[2]}\left(-\frac{2haey^*}{(h+x^*)^2}v^{[1]}v^{[1]}+\frac{2aeh}{(h+x^*)^2}v^{[1]}v^{[2]}+\frac{2lpz^*}{(l+y^*)^3}v^{[2]}v^{[2]}-\frac{2pl}{(l+y^*)^2}v^{[2]}\right)+\frac{2 ck}{(k+y^*)^2}v^{[1]}v^{[2]}+\frac{2lpq}{(l+y^*)^2}v^{[2]}-\left(\frac{2ckx^*}{(k+y)^3}+\frac{2lpqz^*}{(l+y^3)^3}\right)v^{[2]}v^{[2]}\neq 0$}
\end{center}
Then By Sotomayor theorem [11], the system (1) experiences a saddle-node bifurcation at the equilibrium
point $E^*$ as the parameter $b$ varies through the bifurcation value $b=\tilde{b}$.
\end{rem}
We now investigate Hopf bifurcation around $E^*$. We consider $b$ as a bifurcation parameter and define
$$g(b)=a_1(b)a_2(b)-a_3(b)$$
Note that if $g(b)=0$, then $b=\bar{b}=-\frac{a_1 a_2+A_{11}A_{22}A_{33}-A_{12}A_{21}A_{33}-A_{11}A_{23}A_{32}}{A_{12}A_{23}}-\frac{cy^*}{k+y^*}$. Now, we will show that the Hopf bifurcation occurs for the system (1) at $ b = \bar{b} $.

\begin{teo} If there exists $b =\bar{b}$. Then the positive equilibrium point $E^*=(x^*(b), y^* (b),$ $z^*(b))$ is locally stable if $b>\bar{b}$, but it is unstable for $b<\bar{b}$ and a Hopf bifurcation of periodic solution occurs at $b=\bar{b}$.
\end{teo}
\begin{proof}
We assume that $E^*$ is locally asymptotically stable, let $$g(b) = a_1(b)a_2(b) - a_3(b).$$ 
Then $a_1(\bar{b})>0$, $g(\bar{b})=0$ and $g'(\bar{b})=A_{12}A_{23}>0$ by a similar argument to the proof of Theorem 4 in [12] the proof is completed.
\end{proof}
\section{Numerical simulations}
In this section, we will make some numerical simulations to verify the results
obtained in section 4 and give examples to illustrate theorems in section 5. In
system (1), we set: 

\begin{figure}[p]
\centering
\includegraphics[scale=.47]{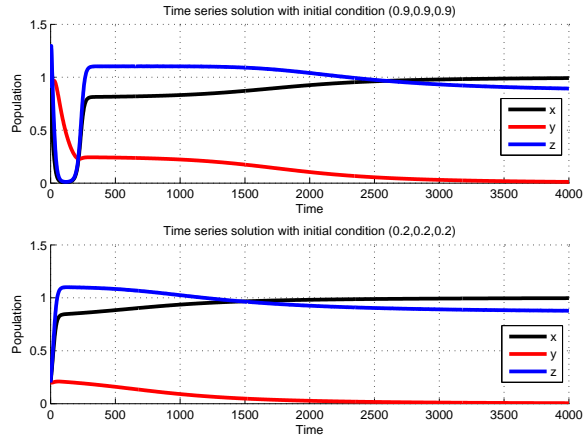}
\caption{Numerical simulation of (1) indicate $E_1$ is locally asymptotically stable.}
\end{figure}

\begin{figure}[p]
\centering
\includegraphics[scale=.47]{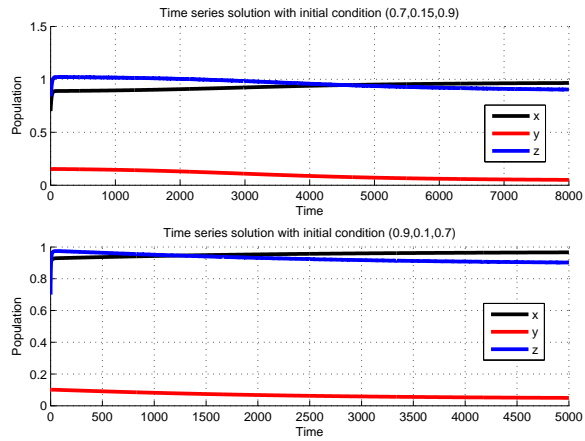}
\caption{Numerical simulation of (1) indicate $E^*\approx$(0.9707,0.0431,0.8908) is locally asymptotically stable.}
\end{figure}

\begin{figure}[p]
\centering
\includegraphics[scale=.47]{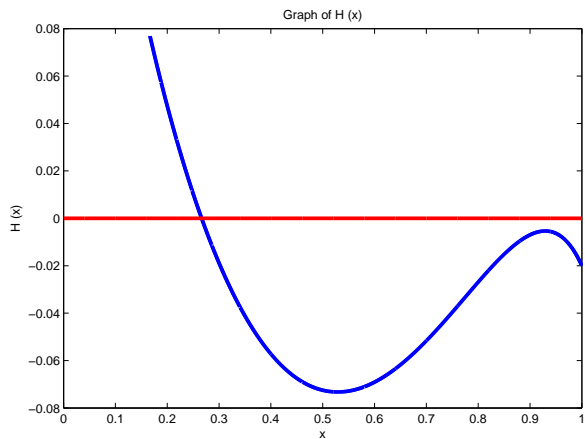}
\caption{Graph of $ H (x) $ indicate that $E^*\approx$(0.2664,0.5622,0.4147) is the only equilibrium point in the interior of $ \Omega $.}
\end{figure}

\begin{figure}
\centering
\includegraphics[scale=.47]{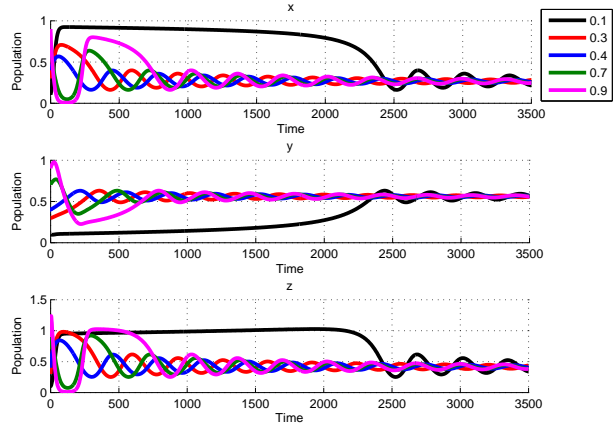}
\caption{Numerical simulation of (1) indicate $E^*\approx$(0.2664,0.5622,0.4147) is globally asymptotically stable.}
\end{figure}

\begin{figure}
\centering
\includegraphics[scale=.47]{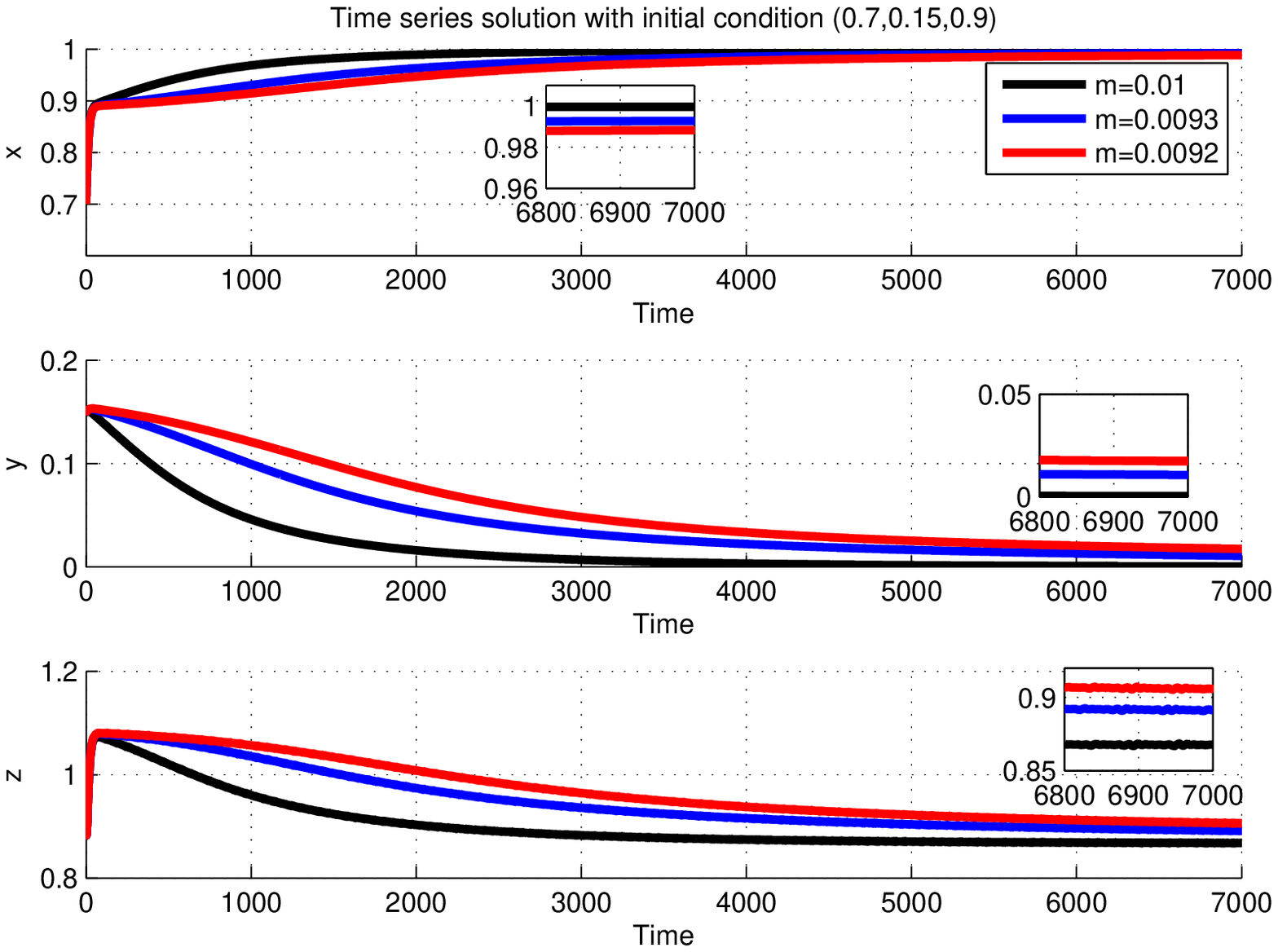}
\caption{Solutions of (1) shows transcritical bifurcation around the equilibrium point $E_1$ when $m =$ 0.00933.}
\end{figure}

\begin{figure}
\centering
\includegraphics[scale=.47]{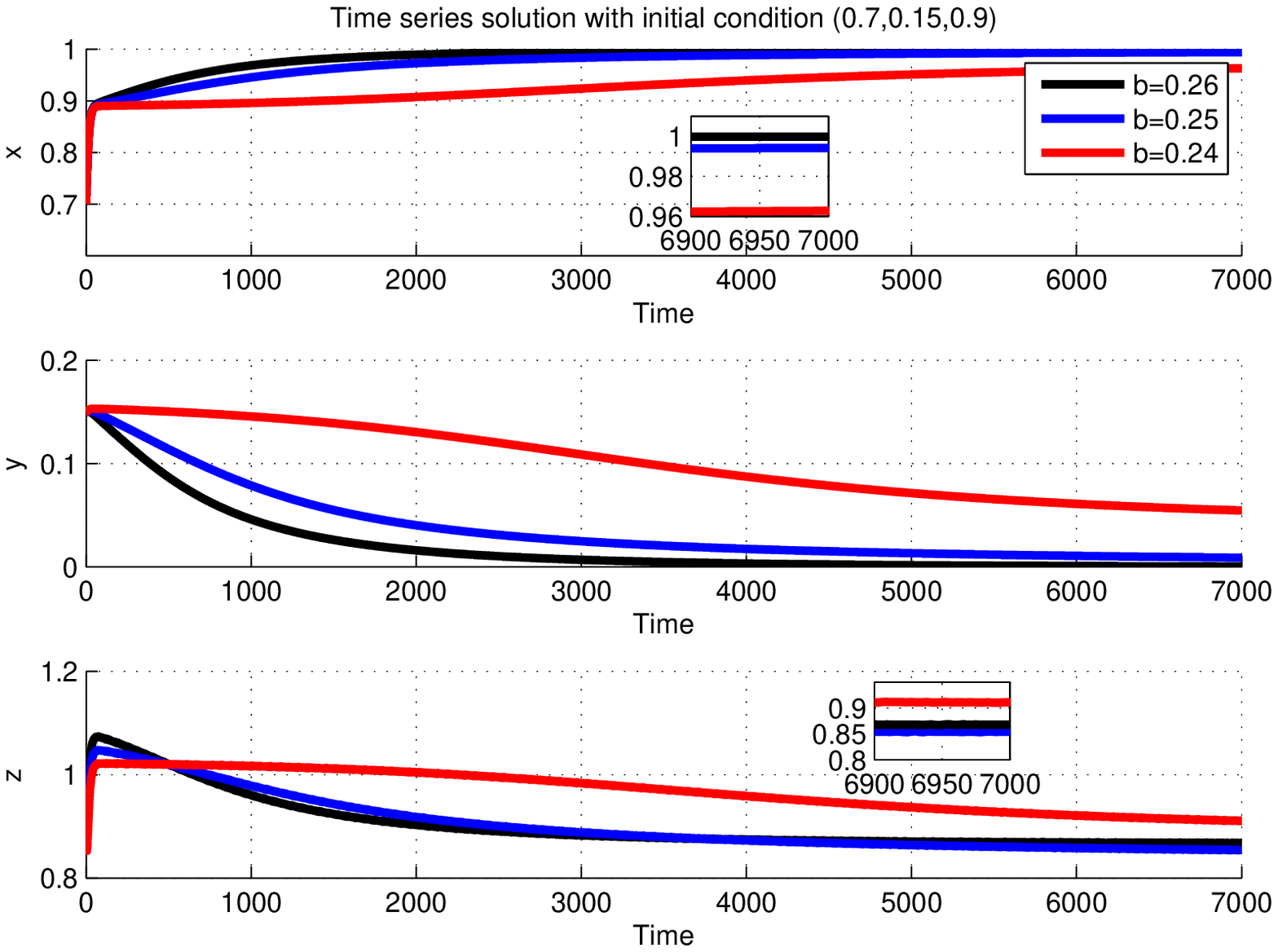}
\caption{Solutions of (1) shows transcritical bifurcation around the equilibrium point $E_1$ when $b =$ 0.25.}
\end{figure}

\begin{figure}
\centering
\includegraphics[scale=.7]{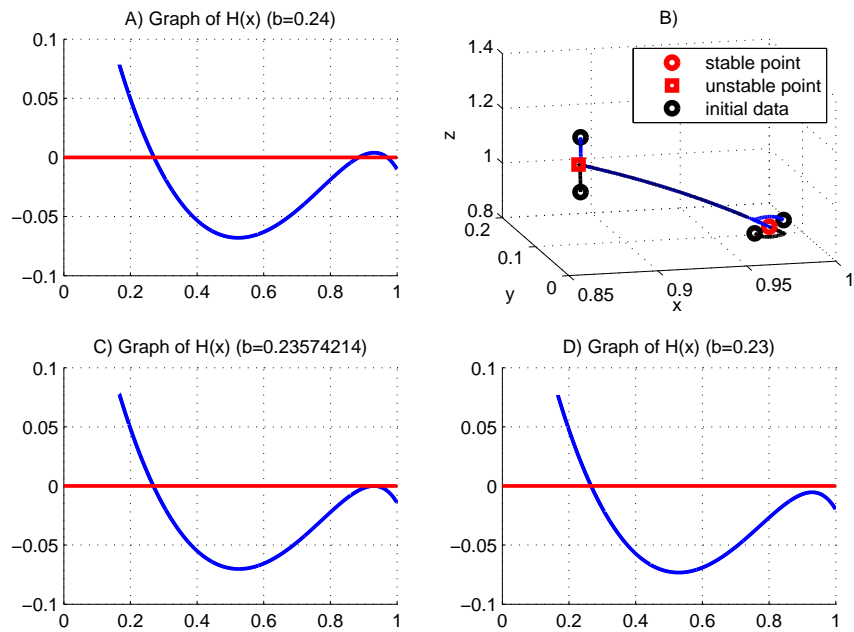}
\caption{In A) we observe that for $ b = 0.24$ there are 2 equilibrium points, in (B) it is observed that one of the points is stable and the other is unstable, in (C and D) we notice that Saddle-node bifurcation occurs in $ b =0.23574214 $}
\end{figure}

\begin{figure}
\centering
\includegraphics[scale=.7]{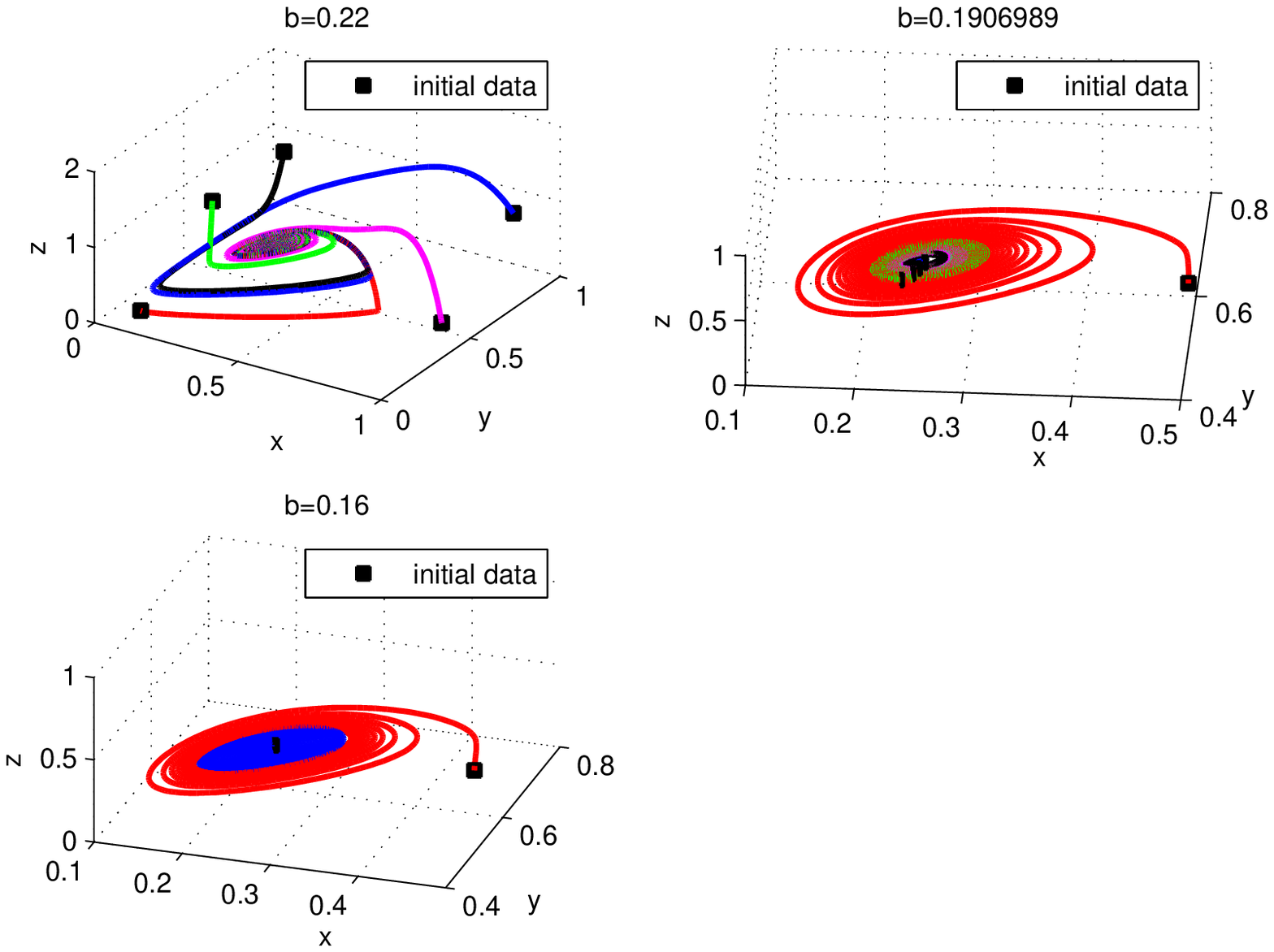}
\caption{Hopf bifurcation occurs at $b =\bar{b} \approx$0.1906989.}
\end{figure}

\begin{center}
$r=$0.1, $K=$1, $h=$0.5, $a=$0.1, $e=$0.4, $m=0.01$, $p=$0.01, $l=$0.5, $c=$0.44, $k=$0.5, $q=$0.5 and $n=$0.3. 
\end{center}
\begin{eje}
In system (1), we set $b=$0.26, then $\frac{pbK}{nl} +m =$0.0273 and $\frac{aeK}{h+K}=$0.0267. By theorem 2, $E_1 = (K, 0, \frac{bK}{n}) \approx$(1,0,0.8667) is locally asymptotically stable, see Figure 3.
\end{eje}

\begin{eje}
In system (1), we set $b =$0.24, then $\frac{pbK}{nl} +m =$0.026 and $\frac{aeK}{h+K} =$0.0267. Then   $a_1 =$0.3934, $a_2 =$0.0286, $a_3 =$ 1.16$\times 10^{-5}$ and $a_1a_2-a_3 =$0.0112. By Remark 1, $E^*\approx$(0.9707,0.0431,0.8908) is locally asymptotically stable, see Figure 4.
\end{eje}

\begin{eje}
In system (1), we set $K=$1, $b=$0.23, $c=$0.44, $m=$0.01 and $e=$0.4. We have that $ H (x) $ has a only root in the interval $\left(\frac{mh}{ae-m},1 \right)$ (see Figure 5), then $E^*\approx$(0.2664,0.5622,0.4147) is the only equilibrium point in the interior of $ \Omega $. Besides, we choose $\eta=$0.2, $\alpha=$4 and $\beta=\zeta=1$, then $\frac{aeK}{h+K} =$0.0267, $\frac{pbK}{nl} +m=0.0253$, $N_{11}=$0.1027, $N_{12}=$-0.0286, $N_{21}=$1.0561, $N_{22}=$-0.2395, $N_{23}=$-0.0286,  $N_{31}=$-0.3557, $N_{32}=$0.0408, $N_{33}=$-0.2787 and $L=\{\text{-0.0116,-0.0040,-0.3265}\}$. By theorem 5, $E^*$ is globally asymptotically stable, see Figure 6.
\end{eje}

\begin{eje}
In system (1), we set $b =$0.26. If we increase the value of the parameter $m$ and keeping all other parameters values fixed, we observe that transcritical bifurcation arises when $m^* =$ 0.00933, see Figure 7.
\end{eje}

\begin{eje}
In system (1). If we increase the value of the parameter $b$ and keeping all other parameters values fixed, we observe that transcritical bifurcation arises when $b^*=$ 0.25, see Figure 8.
\end{eje}

\begin{eje}
In system (1) we observe that if $b=0.24$ then  $E^*_1\approx$(0.9707,0.0431,0.8908) is locally asymptotically stable and $E^*_2\approx$(0.8852,0.1591,1.0256) is unstable. Also if we increase the value of the parameter $b$ and keeping all other parameters values fixed, we observe that saddle-node bifurcation occurs at $b =\tilde{b} \approx$0.23574214, see Figure 9.
\end{eje}

\begin{eje}
In system (1). If we increase the value of the parameter $b$ and keeping all other parameters values fixed, we observe that Hopf bifurcation arises when $b =$0.1906989, see Figure 10.
\end{eje}
\section{Discusssion}
In this paper we have considered a mathematical model to describe the tritrophic interaction between crop, pest and the pest natural enemy, in which the release of
Volatile Organic Compounds (VOCs) by crop is explicitly taken into account. We
obtained three equilibrium points:
\begin{itemize}
\item[•]The ecosystem collapse is at point $E_0 = (0, 0, 0)$.
\item[•]The aphid-free is at point $E_1 =\left(K,0,\frac{b}{n}K \right)$.
\item[•]The coexistence is at point $E^*$.
\end{itemize}
We have investigated the topics of existence and non-existence of various equilibria and their stabilities. More precisely, we have proved the following:
\begin{itemize}
\item[•]$E_0 = (0, 0, 0)$ is unstable.
\item[•]If $\frac{aeK}{h+K}<\frac{pbK}{nl}+m$, then $E_1 = (K, 0,\frac{b}{n}K)$
is locally asymptotically stable.
If $\frac{aeK}{h+K}>\frac{pbK}{nl}+m$, then $E_1$ is unstable.
\item[•]$E^*$ it is locally asymptotically stable if $r+\frac{ax^*y^*}{(h+x^*)^2}<2\frac{rx^*}{K}+\frac{ay^*}{h+x^*}$, $\frac{aex^*}{h+x^*}<\frac{lpz^*}{(l+y^*)^2+m}$ and $\frac{pqy^*}{l+y^*}<n$ and $-A_{11}A_{22}A_{33}-A_{12}A_{23}A_{31} + A_{12}A_{21}A_{33} +
A_{11}A_{23}A_{32} > 0$ or $a_i > 0$ for $i =$ 1, 2, 3 and $a_1a_2 - a_3 > 0$.
\end{itemize}
We also show the global stability of the positive equilibrium by high-dimensional Bendixson criterion. We used the Sotomayor's theorem to ensure the existence of saddle-node bifurcation and transcritical bifurcation (this type of bifurcation transforms a herbivore free equilibrium point from stable situation to a unstable). In this paper, we have chosen the parameters $m$ and $b$ arbitrarily to obtain this type of bifurcation. From Hopf bifurcation analysis we observed that $b$ (the attraction constant due to VOCs.) decreasing destabilizes the system.

Thus, $ b $ is an important parameter for our model, because the aphid-free point ($E_1$) is locally asymptotically stable for $b$ sufficiently large. We also found three critical values for b ($ b^* $, $ \tilde{b} $ and $\bar {b}$) and we got that
\begin{itemize}
\item[•] If $b>b^*$, then $E_1$ is locally asymptotically stable and If $b<b^*$, then $E_1$ is unstable.
\item[•] If $b=b^*$, then a transcritical bifurcation occurs.
\item[•] If $\tilde{b}<b<b^*$, then there are 2 positive equilibrium points $ E_1 ^ * $ (locally asymptotically stable) and $ E_2 ^ * $ (unstable).
\item[•] If $b=\tilde{b}$, then a saddle-node bifurcation occurs.
\item[•] $\bar{b}<b<\tilde{b}$, then there is only one positive equilibrium point $E^*$ that is globally asymptotically stable. 
\item[•] $b=\bar{b}$,  then a Hopf bifurcation occurs.
\item[•] $b<\bar{b}$, then there is only one positive equilibrium point $E^*$ that is unstable. 
\end{itemize}
Therefore,  VOCs  possess a beneficial effect on the environment since their release may be able to stabilize the model dynamics. This could reduce the use of synthetic pesticides. 

\textbf{Acknowledgments} This work was supported by Sistema Nacional de Investigadores (15284) and Conacyt-Becas.
\section*{References}
\begin{itemize}
\item[1)]B. Buonomo, F. Giannino, S. Saussure and E. Venturino. Effects of limited volatiles release by plants in tritrophic interactions. Mathematical Biosciences and Engineering, 16(2019), 3331-3344.
\item[2)]F. Brilli, F. Loreto and I. Baccelli. Exploiting Plant Volatile Organic Compounds (VOCs) in Agriculture to Improve Sustainable Defense Strategies and Productivity of Crops. Frontiers In Plant Science, 10(2019):264.
\item[3)]J. Takabayashi and M. Dicke. Plant—carnivore mutualism through herbivore-induced carnivore attractants. Trends In Plant Science, 1(1996), 109-113. 
\item[4)]L. Tollsten, P. Mller. Volatile organic compounds emitted from beech
leaves. Phytochemistry, 43(1996), 759-762.
\item[5)]D. Mukherjee. Dynamics of defensive volatile of plant modeling tritrophic
interactions. International Journal of Nonlinear Science 25(2018), 76-86.
\item[6)]R. Mondal, D. Kesh and D. Mukherjee. Role of Induced Volatile Emission Modelling Tritrophic Interaction. Differential Equations And Dynamical Systems (2019). 
\item[7)]R. Mondal, D. Kesh and D. Mukherjee. Influence of induced plant volatile and refuge in tritrophic model. Energy Ecology and Environment, 3(2018), 171–184
\item[8)] M. Li and J. Muldowney. A geometric approach to global-stability problems.
SIAM Journal on Mathematical Analysis, 27(1996), 1070-1083.
\item[9)]H. I. Freedman and P. Waltman. Persistence in models of three interacting
predator-prey populations. Mathematical Biosciences, 68(1984), 213-231.
\item[10)]G. Butler, H. Freedman and P. Waltman. Uniformly persistent systems.
 Proceedings of the American Mathematical Society, 96(1986), 425-430.
\item[11)]L. Perko. Differential equations and dynamical systems. Springer Science \&
Business Media, 7(2013).
\item[12)]D. Mukherjee. The effect of refuge and immigration in a predator–prey system in
the presence of a competitor for the prey. Nonlinear Analysis: Real World Applications, 31(2016), 277–287.
\end{itemize}
\end{document}